\documentclass{amsart}

\usepackage{color}
\usepackage{hyperref}

\newtheorem{theorem}{Theorem}
\newtheorem{lemma}[theorem]{Lemma}
\newtheorem{proposition}[theorem]{Proposition}

\newtheorem{definition}[theorem]{Definition}
\newtheorem{remark}[theorem]{Remark}
\newtheorem{example}[theorem]{Example}
\newtheorem{corollary}[theorem]{Corollary}

\title{Commutators on generalized block-triangular algebras}

\author{Pedro Souza Fagundes}
\address{Departamento de Matem\'atica, Universidade Federal de S\~ao Carlos, 13565-905, S\~ao Carlos, S\~ao Paulo, Brazil}
\email{pedrofagundes@ufscar.br}

\author{Thiago Castilho de Mello}
\address{Universidade Federal de S\~ao Paulo, Instituto de Ci\^encia e Tecnologia}
\email{tcmello@unifesp.br}

\begin{document}

\begin{abstract}
    The characterization of commutators in associative algebras is a classical problem in ring theory. In this paper, we address this problem for the natural class of generalized block-triangular algebras. To this end, we introduce a new invariant: \emph{the multitrace} of an arbitrary element in an associative unital algebra, and prove that in a generalized block-triangular algebra, an element is a commutator if and only if its multitrace vanishes. As a consequence, we show that the set of commutators is closed under addition in these algebras. Our main result extends the classical Albert–Muckenhoupt–Shoda theorem for full matrix algebras to the broader setting of generalized block-triangular algebras.
\end{abstract}



\maketitle

\section{Introduction}

Additive commutators are of fundamental importance in understanding the structure of an algebra. For instance, a semiprimitive ring is commutative if and only if all additive commutators lie in its center \cite[Theorem 12.11]{Lam}. It is also well known that a division ring is generated as a noncommutative division algebra over its center by its set of commutators \cite[Corollary 3.16]{Lam}.

In this paper, we study commutators in a class of finite-dimensional algebras. The well-known Albert-Muckenhoupt-Shoda Theorem (see \cite{Shoda} and \cite{AM}) states that for an arbitrary field $K$, a matrix $a \in M_n(K)$ is a commutator if and only if $a$ has trace zero, i.e., $a \in \mathfrak{sl}_n(K)$ (see also \cite{Trace0} for a simplified proof). A similar result holds for the algebra of upper triangular matrices: a matrix $a \in UT_n(K)$ is a commutator if and only if all its diagonal entries are zero.

These two algebras are special cases of a more general class known as algebras of block-triangular matrices. Denote by $UT(d_1,\dots,d_m, K)$ the algebra of block-triangular matrices over a field $K$. This is the subalgebra of $M_n(K)$, where $n = d_1 + \cdots + d_m$, consisting of matrices with zero entries below the diagonal blocks of sizes $d_1, \dots, d_m$.

For the algebra of block-triangular matrices, an analogous result holds: a matrix $a \in UT(d_1,\dots,d_m; K)$ is a commutator if and only if the trace of each of its diagonal blocks is zero. This is a particular case of our main result (Theorem \ref{main theorem}) when $K$ is algebraically closed; however, it holds over infinite fields with essentially the same proof.

As a consequence, any sum of commutators in $UT(d_1,\dots,d_m; K)$ is itself a commutator. This fact is related to problems concerning the images of polynomials on algebras (see, for instance, \cite{survey} for problems related to polynomial images and the Lvov-Kaplansky conjecture). Indeed, it implies that for these algebras, the image of the commutator polynomial is a vector space. On the other hand, sum of commutators is not always again a commutator. For instance in \cite{Makar-Limanov} Makar-Limanov constructed an example of a division skew-field $D$ where every element in $D$ is a sum of commutators, but $D$ contains elements which are not commutators.  

In this paper, we define a broader class of algebras called generalized block-triangular algebras. This class contains the algebras $UT(d_1,\dots,d_m, K)$ over an algebraically closed field of characteristic zero. We show that commutators in these algebras can be characterized in a similar way, using the notion of multitrace zero elements. This definition is based on the Wedderburn-Malcev decomposition of a finite-dimensional associative algebra, which we present in the next section.

\section{Multitrace zero elements}

Let $K$ be an algebraically closed field of characteristic zero, and let $A$ be a finite-dimensional unital associative algebra over $K$. By the well-known Wedderburn-Malcev Theorem, there exists a semisimple subalgebra $B$ of $A$ such that
\[
A = B \oplus \operatorname{rad}(A)
\]
where $\operatorname{rad}(A)$ is the Jacobson radical of $A$ and the direct sum is of vector spaces.

Moreover, since $B$ is semisimple and $K$ is algebraically closed, the Wedderburn-Artin Theorem states that $B$ is isomorphic, as an algebra, to a direct sum of matrix algebras:
\[
B \simeq M_{d_1}(K) \oplus \cdots \oplus M_{d_r}(K),
\]
for some $d_1, \dots, d_r \in \mathbb{N}$.

Related to the decomposition of $A$ above, in what follows we introduce the notion of multitrace. It will be based on the concept of a multiset. This is essentially a set equipped with a map to positive integers assigning to each element its multiplicity. We will use the same notations of usual sets to denote multisets. For instance, $\{1, 1, 1,2,5,5\}$ is the multiset with elements $1, 2$ and $5$ in which the multiplicities are respectively $3$, $1$ and $2$.

\begin{definition}
    Let $A$ be a finite-dimensional unital algebra over an algebraically closed field $K$ of characteristic zero, and let $A = B \oplus \operatorname{rad}(A)$ be its Wedderburn-Malcev decomposition. For an element $a = b + j \in A$ with $b \in B$ and $j \in \operatorname{rad}(A)$, the \emph{multitrace} of $a$ is the multiset
    \[
    \operatorname{mtr}(a) = \{ \operatorname{tr}(b_1), \operatorname{tr}(b_2), \dots, \operatorname{tr}(b_r) \},
    \]
    where $b$ is mapped to $(b_1, \dots, b_r) \in M_{d_1}(K) \oplus \cdots \oplus M_{d_r}(K)$ under the isomorphism. We say that $a$ has \emph{multitrace zero} if $\operatorname{tr}(b_i) = 0$ for all $i = 1, \dots, r$. In this case, we denote simply $\operatorname{mtr}(a) = 0$.
\end{definition}

\begin{remark}
The notion of the multitrace is well-define. In particular, the property of having multitrace zero is independent of the choice of the semisimple part $B$ and the order of the components.
\end{remark}
Indeed, the definition depends \emph{a priori} on the choice of the semisimple subalgebra $B$. However, in a multiset, the ordering is irrelevant and by the uniqueness part of the Wedderburn-Malcev theorem \cite[Theorem 72.19]{Curtis-Reiner}, if $A = B' \oplus \operatorname{rad}(A)$ is another decomposition with $B'$ semisimple, there exists an element $r \in \operatorname{rad}(A)$ such that $B' = (1 + r)^{-1}B(1 + r)$; that is, $B'$ is the image of $B$ under an inner automorphism of $A$ given by conjugation by the unit $1 + r$. This automorphism permutes the simple components of $B$ and maps each component onto an isomorphic simple component of $B'$. The restriction of this automorphism to a simple component $M_{d_i}(K)$ of $B$ is an automorphism of $M_{d_i}(K)$ and, by the Skolem-Noether Theorem, it is inner. Since inner automorphisms of matrix algebras preserve traces, the multiset of traces $\{\operatorname{tr}(b_1), \dots, \operatorname{tr}(b_r)\}$ is invariant under the automorphism between $B$ and $B^\prime$. Consequently, the multitrace of an element, and in particular the property of having multitrace zero, is well-defined.

Let us now examine this notion for some known algebras.
\begin{itemize}
    \item If $A = UT_n(K)$ is the algebra of upper triangular matrices over $K$, we have $A = B \oplus \operatorname{rad}(A)$, where $B \simeq K^n$ (the algebra of diagonal matrices) and $\operatorname{rad}(A)$ is the ideal of strictly upper triangular matrices. It is easy to check that for $a \in UT_n(K)$,
    \[
    \operatorname{mtr}_B(a) = 0 \quad \Leftrightarrow \quad a \in \operatorname{rad}(A).
    \]
    \item If $A$ is semisimple, then $\operatorname{rad}(A) = 0$ and $A \simeq M_{n_1}(K) \oplus \cdots \oplus M_{n_r}(K)$. In this case, the multitrace zero elements of $A$ are exactly those mapped to $\mathfrak{sl}_{n_1}(K) \oplus \cdots \oplus \mathfrak{sl}_{n_r}(K)$.
\end{itemize}

Let us have a deeper look at these two examples. To that, let us denote by $c(A)$ the set of commutators in the algebra $A$, i.e., \[c(A) = \{[a,b]\; a, b \in A\}.\] In the first example above, $c(A)$ coincides with $\operatorname{rad}(A)$. In the second one, as a consequence of the Albert-Muckenhoupt-Shoda Theorem, $c(A) = \mathfrak{sl}_{n_1}(K) \oplus \dots \oplus \mathfrak{sl}_{n_r}(K)$. In both cases, the set of commutators is a vector space and can be described as the set of elements with multitrace zero. This discussion naturally leads to the following question: For a finite-dimensional unital algebra $A$, is the set of commutators equal to the set of multitrace zero elements in $A$? The next lemma tells us that we have at least one inclusion.

\begin{lemma}\label{commutator implies multitrace zero}
    Every commutator in $A$ has multitrace zero.
\end{lemma}

\begin{proof}
    Let $A = B \oplus \operatorname{rad}(A)$ and let $[a_1, a_2] \in [A, A]$. Write $a_1 = b_1 + j_1$ and $a_2 = b_2 + j_2$ where $b_1, b_2 \in B$ and $j_1, j_2 \in \operatorname{rad}(A)$. Then
    \[
    [a_1, a_2] = [b_1, b_2] + j
    \]
    for some $j \in \operatorname{rad}(A)$. Since $B$ is semisimple, $[b_1, b_2]$ is mapped into $\mathfrak{sl}_{n_1}(K) \oplus \cdots \oplus \mathfrak{sl}_{n_r}(K)$ under the isomorphism $B \simeq M_{d_1}(K) \oplus \cdots \oplus M_{d_r}(K)$. Therefore, $[a_1, a_2]$ has multitrace zero.
\end{proof}

In general, the converse of Lemma \ref{commutator implies multitrace zero} is not true, as shown in the next example.

\begin{example}\label{example0}
Consider the algebra
\[
    A = \left\{ \begin{pmatrix}
        a & c & d \\
        0 & a & 0 \\
        0 & 0 & b
    \end{pmatrix} \middle| a, b, c, d \in K \right\},
\]
with the usual matrix operations. It is easy to check that
\[
    A = B \oplus \operatorname{rad}(A),
\]
where $B = K(e_{11} + e_{22}) \oplus Ke_{33}$ is a 2-dimensional semisimple subalgebra and $\operatorname{rad}(A) = Ke_{12} + Ke_{13}$.
Since $B \simeq K \times K$, the set of multitrace zero elements of $A$ coincides with $\operatorname{rad}(A)$. On the other hand, a straightforward computation shows that $c(A) = K e_{13}$.
\end{example}

In light of the above example, we reformulate the question as follows: For which classes of algebras the set of  multitrace zero elements coincide with the set of commutators?

We note that in algebras $A$ satisfying the above condition, the sum of two commutators is again a commutator, since the sum of multitrace zero elements is again a multitrace zero element, in particular, the set of commutators $c(A)$ would be a vector space. 

Our goal in this paper is to present a large class of algebras that answers this question positively. But first, we need to introduce our main technique: the module-theoretic version of the Sylvester-Rosenblum equation.

\section{The module-theoretic Sylvester-Rosenblum equation}

We present here the module-theoretic version of the Sylvester-Rosenblum equation. The proof follows essentially the same argument as in \cite{BhatiaRosenthal}, but is formulated in the language of modules. In this section, $K$ is an arbitrary field.

Let $a \in A$ and let $M$ be a left $A$-module. We denote by $L_a$ the left multiplication map on $M$ by $a$; that is,
\[
L_a(x) = a x, \quad \text{for all } x \in M.
\]
Similarly, we define the right multiplication map $R_a$ on a right $A$-module $M$. Note that both $L_a$ and $R_a$ are linear operators on $M$ for all $a \in A$.

Recall that if $A$ and $B$ are $K$-algebras, a $K$-vector space $M$ is an $(A,B)$-bimodule if $M$ is both a left $A$-module and a right $B$-module and satisfies the compatibility condition $(a m) b = a (m b)$ for all $a \in A$, $b \in B$, and $m \in M$.

In the next theorem, we denote by $\operatorname{Spec}(L)$ the spectrum of the linear operator $L$ over a finite-dimensional vector space.

\begin{theorem}\label{sylvester}
Let $A$ and $B$ be finite-dimensional $K$-algebras, and let $M$ be an $(A,B)$-bimodule which is finite-dimensional as a vector space over $K$. Given $a \in A$ and $b \in B$, the equation
\[
a x - x b = c
\]
has a unique solution $x \in M$ for any $c \in M$ if $\operatorname{Spec}(L_a) \cap \operatorname{Spec}(R_b) = \emptyset$ in the algebraic closure of $K$.
\end{theorem}

\begin{proof}
We first rewrite the equation as
\begin{equation}\label{LR}
(L_a - R_b) x = c.
\end{equation}
Since $M$ is a finite-dimensional vector space over $K$, fixing a basis for $M$, equation \eqref{LR} has a unique solution if and only if the linear operator $L_a - R_b$ is invertible; that is, if $(L_a - R_b) x = 0$ implies $x = 0$.

Now, $(L_a - R_b) x = 0$ implies $a x = x b$, and by induction we obtain $a^k x = x b^k$ for all $k \geq 1$. Therefore, $L_a^k(x) = R_b^k(x)$ for all $k \geq 1$, and more generally,
\[
p(L_a)(x) = p(R_b)(x)
\]
for any polynomial $p \in K[t]$. In particular, if $p$ is the characteristic polynomial of $L_a$ and $q$ is the characteristic polynomial of $R_b$, we have by the Cayley-Hamilton Theorem
\begin{equation}\label{pRb}
p(R_b)(x) = p(L_a)(x) = 0.
\end{equation}

Since $\operatorname{Spec}(L_a) \cap \operatorname{Spec}(R_b) = \emptyset$, the polynomials $p$ and $q$ are coprime. Thus, there exist polynomials $r$ and $s$ with coefficients in $K$ such that $p r + q s = 1$. Applying this identity to $R_b$, we get
\[
p(R_b) r(R_b) + q(R_b) s(R_b) = I,
\]
but, again by the Cayley-Hamilton Theorem, $q(R_b) = 0$, so $p(R_b) r(R_b) = I$. In particular, the linear operator $p(R_b)$ is invertible, so equation \eqref{pRb} yields $x = 0$. Therefore, $L_a - R_b$ is invertible, and the uniqueness of $x$ follows.
\end{proof}

Next, we present an application of the module-theoretic Sylvester-Rosenblum equation addressed to the L'vov-Kaplansky conjecture, which might be of independent interest.

\subsection{Images of polynomials on triangular algebras}

We start this section by recalling that if $f(x_1,\dots,x_n) \in K\langle x_1,\dots,x_n \rangle$ is a polynomial in the free associative algebra of rank $n$, then its image on an associative $K$-algebra $A$, denoted by $f(A)$,  is the set of all evaluations of $f$ on $A$.

As an application of Theorem \ref{sylvester}, we show that the image of a multilinear polynomial of degree two on a finite-dimensional triangular algebra is always a vector space. Recall that a $K$-algebra $T = A \oplus B \oplus M$ is a triangular algebra if there exist algebras $A$ and $B$ and an $(A,B)$-bimodule $M$ such that $T$ is isomorphic to the algebra $\begin{pmatrix}
    A & M \\ 0 & B
\end{pmatrix}$ with its natural matrix operations.

\begin{theorem}
Let $T = A \oplus B \oplus M$ be a finite-dimensional triangular algebra with unit over an infinite field $K$ and let $f \in K\langle x,y \rangle$ be a multilinear polynomial of degree two. Then $f(T) = f(A) \oplus f(B) \oplus M$.
\end{theorem}

\begin{proof}
    If $f = \alpha xy + \beta yx$ with $\alpha + \beta \neq 0$, then $f(T) = T$, since $1 \in T$. So we may assume that $f = [x,y]$. Take $[a_1,a_2] + [b_1,b_2] + c \in f(A) \oplus f(B) \oplus M$. Let $m \in M$ be an arbitrary element and evaluate $x = a_1 + b_1$ and $y = a_2 + b_2 + m$. Then we have 
    \[
    [x,y] = [a_1,a_2] + [b_1,b_2] + a_1 m - m b_1. 
    \]
    Since $\operatorname{Spec}(L_{a_1})$ and $\operatorname{Spec}(R_{b_1})$ are finite and $K$ is infinite, there exists $\alpha \in K$ such that $\operatorname{Spec}(L_{a_1}) \cap \operatorname{Spec}(R_{b_1 + \alpha \cdot 1_B}) = \emptyset$. Taking $x = a_1 + b_1 + \alpha \cdot 1_B$ instead, we have
    \[
    [x,y] = [a_1,a_2] + [b_1,b_2] + a_1 m - m (b_1 + \alpha \cdot 1_B).
    \]
    Now it is enough to apply Theorem \ref{sylvester} to realize $c$ as $a_1 m - m (b_1 + \alpha \cdot 1_B)$ for some $m \in M$ (which depends on $c$). This shows that $f(A) \oplus f(B) \oplus M \subset f(T)$, and since the reverse inclusion is obvious, we are done. 
\end{proof}



This result provides a complete description of the image of multilinear polynomials of degree two on finite-dimensional triangular algebras. The question for higher degrees is significantly more complex, and as seen in the first example of this paper, the image of a polynomial identity may not simply be the direct sum of its images on the diagonal blocks and the off-diagonal bimodule.
For instance, in the algebra from Example \ref{example0}, the image of the polynomial $f(x_1,x_2,x_3,x_4)=[x_1,x_2][x_3,x_4]$ is strictly contained in $f(A)\oplus f(B)\oplus M$.

\section{Generalized block-triangular algebras}

Our goal in this section is to prove that the sum of two commutators is again a commutator for a class of algebras that includes, in particular, the block-triangular algebras. For this purpose, let us recall that if $A$ is an associative unital $K$-algebra, we can write $A = B \oplus \operatorname{rad}(A)$ where $B \simeq M_{n_1}(K) \oplus \cdots \oplus M_{n_r}(K)$. Also, from the nilpotency of $\operatorname{rad}(A)$, it follows that $1_A = 1_B$. Now, let $e_i \in M_{n_i}(K)$ be the $n_i \times n_i$ identity matrix, $i = 1, \dots, r$. Then
\[
1_A = 1_B = e_1 + \cdots + e_r.
\]
Moreover, $\{e_1, \dots, e_r\}$ is a set of orthogonal idempotents of $A$. We can therefore decompose the algebra $A$ as
\[
A = B \oplus \sum_{1 \leq i, j \leq r} e_i \operatorname{rad}(A) e_j.
\]
For the following definition, we denote $e_i \operatorname{rad}(A) e_j$ simply by $\operatorname{rad}(A)_{ij}$.

\begin{definition}\label{generalized block-triangular algebra}
    Let $K$ be an algebraically closed field of characteristic zero. A finite-dimensional $K$-algebra $A$ is called a \emph{generalized block-triangular algebra} if $\operatorname{rad}(A)_{ij} = 0$ for all $i \geq j$.
\end{definition}

The block-triangular matrix algebra $UT(d_1,\dots,d_m;K)$ is the most natural example of a generalized block-triangular algebra. The next example shows that Definition generalized block-triangular algebra includes a large class of triangular algebras.

\begin{example}\label{example of generalized block-triangular algebra}
Let $A_1$ and $A_2$ be semisimple algebras and $M$ be an $(A_1,A_2)$-bimodule, and consider the triangular algebra
    \[
    T = \begin{pmatrix}
        A_1 & M \\
        0 & A_2
    \end{pmatrix}.
    \] 
By the semisimplicity of both $A_1$ and $A_2$, we write
\[
1_{A_1} = e_1 + \cdots + e_k \quad \mbox{and} \quad 1_{A_2} = e_{k+1} + \cdots + e_r,
\]
where $\{e_i \in A_1 \mid i=1,\dots,k\}$ and $\{e_i \in A_2 \mid i=k+1,\dots,r\}$ are sets of orthogonal idempotents. Hence the unity $1$ of $T$ is such that
\[
1 = 1_{A_1} + 1_{A_2} = e_1 + \cdots + e_r,
\]
and moreover, we have the Peirce decomposition:
\[
T = 1_{A_1} T 1_{A_1} \oplus 1_{A_2} T 1_{A_2} \oplus 1_{A_1} T 1_{A_2}.
\]
Since $\operatorname{rad}(T) = 1_{A_1} T 1_{A_2}$, setting $B = 1_{A_1} T 1_{A_1} \oplus 1_{A_2} T 1_{A_2}$ we have that $B \simeq A_1 \oplus A_2$ is a semisimple subalgebra of $T$ and $T = B \oplus \operatorname{rad}(T)$. Finally, we notice that for the idempotents $e_1, \dots, e_r$,
\[
\operatorname{rad}(T)_{ij} = e_i \operatorname{rad}(T) e_j = e_i (1_{A_1} T 1_{A_2}) e_j.
\]
If $i \geq j$, then either both $e_i$ and $e_j$ belong to $A_1$, both belong to $A_2$, or $e_i \in A_2$ and $e_j \in A_1$. In all cases, at least one of $e_i 1_{A_1}$ or $1_{A_2} e_j$ is zero. Consequently, $\operatorname{rad}(T)_{ij} = 0$. We conclude that $T$ is a generalized block-triangular algebra. 
\end{example}

We are ready for the main result of the paper.

\begin{theorem}\label{main theorem}
    Let $A$ be a generalized block-triangular algebra. Then an element $a \in A$ is a commutator in $A$ if and only if $a$ has multitrace zero.   
\end{theorem}

\begin{proof}
By Lemma \ref{commutator implies multitrace zero}, every commutator in $A$ has multitrace zero. For the converse, let $A  = B \oplus \operatorname{rad}(A)$, be the Wedderburn-Malcev decomposition of $A$. We proceed by induction on the number $r$ of simple components of $B$. 

If $r = 1$, we note that $A = B \simeq M_{d_1}(K)$, since in this case $e_1 = 1_A$ and then $\operatorname{rad}(A) = \operatorname{rad}(A)_{11} = 0$. The result now follows from the classical Albert-Muckenhoupt-Shoda Theorem.

Now assume $r > 1$ and that the result holds for every generalized block-triangular algebra with fewer than $r$ simple components in its semisimple part.

Let $B_0$ be the subalgebra of $B$ isomorphic to $M_{d_1}(K) \oplus \cdots \oplus M_{d_{r-1}}(K)$ and denote by $\operatorname{rad}(A)_0$ the ideal $\sum_{1 \leq i < j \leq r-1} \operatorname{rad}(A)_{ij}$ of $A$. Then $A_0 = B_0 \oplus \operatorname{rad}(A)_0$ is a subalgebra of $A$, and moreover,
\[
A = A_0 \oplus B_r \oplus \sum_{i=1}^{r-1} \operatorname{rad}(A)_{ir},
\]
where $B_r$ is the subalgebra of $B$ isomorphic to $M_{n_r}(K)$. We claim that $A_0$ is itself a generalized block-triangular algebra. To prove this, first note that $B_0$ is semisimple. Second, $\operatorname{rad}(A)_0$ coincides with the Jacobson radical of $A_0$. Clearly, $\operatorname{rad}(A)_0$ is a nilpotent ideal of $A_0$. It is also maximal with this property because for any other nilpotent ideal $I$ of $A_0$, the image $\pi(I)$ under the projection $\pi: A_0 \to B_0$ is a nilpotent ideal of the semisimple algebra $B_0$, forcing $\pi(I) = 0$; hence $I \subset \operatorname{rad}(A)_0$. Finally, since $1_{B_0} = e_1 + \cdots + e_{r-1}$, we get for all $k \geq l$ that
\[
e_k \operatorname{rad}(A)_0 e_l = \sum_{1 \leq i < j \leq r-1} e_k e_i \operatorname{rad}(A) e_j e_l = 0,
\]
because if $e_k e_i \neq 0$ then $k = i$, and if $e_j e_l \neq 0$ then $j = l$, but then $i = k \geq l = j$, which contradicts the condition $i < j$ in the sum. This proves the claim.

Now let $a \in A$ be a multitrace zero element. Write $a = b + j$ where $b \in B$ is mapped to $\mathfrak{sl}_{n_1}(K) \oplus \cdots \oplus \mathfrak{sl}_{n_r}(K)$ and $j \in \operatorname{rad}(A)$. Decompose $b = b_0 + c$ where $b_0$ is a sum of traceless matrices in $M_{n_1}(K) \oplus \cdots \oplus M_{n_{r-1}}(K)$ and $c$ is a traceless matrix in $M_{n_r}(K)$. Therefore, we can write $a = b_0 + j_0 + c + j_r$ for some $j_0 \in \operatorname{rad}(A)_0$ and $j_r \in \sum_{i=1}^{r-1} \operatorname{rad}(A)_{ir}$.

Since $a_0 = b_0 + j_0 \in A_0$ is a multitrace zero element of $A_0$, the induction hypothesis implies $a_0 = [a_0', a_0'']$ for some $a_0', a_0'' \in A_0$. By the Albert-Muckenhoupt-Shoda Theorem, we can write $c = [c', c'']$ for some $c', c'' \in B_r$. Set $x = a_0' + c'$. To finish the proof, we will show that there exists $z \in \sum_{i=1}^{r-1} \operatorname{rad}(A)_{ir}$ such that $y = a_0'' + c'' + z$ satisfies $a = [x, y]$.

Notice that if such a $z$ exists, then
\begin{align*}
[x, y] &= [a_0' + c', a_0'' + c'' + z] \\
&= [a_0', a_0''] + [c', c''] + a_0' z - z c' \\
&= a_0 + c + (a_0' z - z c').
\end{align*}

Since $a = a_0 + c + j_r$, it is sufficient to find $z$ such that
\begin{equation}\label{sylvesteq}
j_r = a_0' z - z c'.
\end{equation}

We view equation \eqref{sylvesteq} as a Sylvester equation in the variable $z$, which lies in the $(B_0, B_r)$-bimodule $\sum_{i=1}^{r-1} \operatorname{rad}(A)_{ir}$. By Theorem \ref{sylvester}, a unique solution exists if the linear operators $L_{a_0'}$ and $R_{c'}$ on this bimodule have disjoint spectra. This might not be true in general, but since $K$ is infinite, we can choose $\lambda \in K$ such that $L_{a_0' + \lambda \cdot 1}$ and $R_{c'}$ have disjoint spectra. Note that $a_0 = [a_0' + \lambda \cdot 1, a_0'']$ still holds. Redefining $x = a_0' + \lambda \cdot 1 + c'$, we can then apply Theorem \ref{sylvester} to solve the equation $j_r = (a_0' + \lambda \cdot 1) z - z c'$ for $z$. With this $z$, we obtain $a = [x, y]$ as desired.
\end{proof}

From Theorem \ref{main theorem} and the fact that the sum of multitrace zero elements is again a multitrace zero element, we have the following consequence.

\begin{corollary}
    In a generalized block-triangular algebra, the sum of two commutators is again a commutator.
\end{corollary}

It is also a consequence of Theorem \ref{main theorem} that $\operatorname{rad}(A)$ is contained in the set $c(A)$ of commutators of $A$ for generalized block-triangular algebras. More generally, we have $e_i \operatorname{rad}(A) e_j \subset c(A)$ for $i \neq j$, since $e_i a e_j = [e_i a e_j, e_j]$ for all $a \in \operatorname{rad}(A)$. On the other hand, if $i = j$, it is not always true that $e_i \operatorname{rad}(A) e_i \subset c(A)$, as shown in the example below. This example also provides an instance of a finite-dimensional associative algebra with multitrace zero elements that are not commutators.

\begin{example}
     Let $A = M_2(K)[x]$ be the algebra of polynomials in one indeterminate $x$ with coefficients in $M_2(K)$, where $x$ commutes with the elements of $M_2(K)$. Let $I$ be the ideal of $A$ generated by $x^2$. Then
    \[
    A/I = \{a + b x + I \mid a, b \in M_2(K)\}.
    \]
    Notice that 
    \[
    A/I = B \oplus \operatorname{rad}(A/I),
    \]
    where $B = M_2(K)$ and $\operatorname{rad}(A/I) = \{b x + I \mid b \in M_2(K)\}$. Moreover, $A/I$ is an $8$-dimensional unital associative algebra. Since $B$ is simple, we set $e_1 = 1$, and then $e_1 \operatorname{rad}(A/I) e_1 = \operatorname{rad}(A/I)$. On the other hand, we have that 
    \[
    [A/I, A/I] \subset \mathfrak{sl}_2(K) + \mathfrak{sl}_2(K) x + I.
    \]
    This shows that $e_1 \operatorname{rad}(A/I) e_1 = \operatorname{rad}(A/I) \not\subset [A/I, A/I]$.
\end{example}

\section{Final remarks}

It is a natural question to measure how different the sets $c(A)$ and $A$ are. An interesting measure for this is the dimension of the quotient space $A / [A, A]$. Recall that $[A, A]$ denotes the linear span of the set of commutators $c(A)$. If $c(A)$ is a vector space, it coincides with $[A, A]$. It is worth mentioning that in the paper \cite{KB_commutant}, the author proves that if $A$ is a PI-algebra, then $A\neq [A,A]$.

\begin{proposition}\label{dimension bound}
    Let $A$ be a finite-dimensional associative unital algebra over an algebraically closed field $K$, and let $r$ be the number of simple components in the Wedderburn decomposition of the semisimple algebra $A/\operatorname{rad}(A)$ (i.e., $A/\operatorname{rad}(A) \cong M_{n_1}(K) \times \cdots \times M_{n_r}(K)$). Then
    \[
    \dim_K \dfrac{A}{[A, A]} \geq r.
    \]
\end{proposition}

\begin{proof} The proof follows the argument of  \cite[Theorem 7.17]{Lam}. 
    Let $\pi: A \to A/\operatorname{rad}(A)$ be the natural projection map. This map induces a well-defined linear map between the quotient spaces:
    \[
    \overline{\pi}: \dfrac{A}{[A, A]} \to \dfrac{A/\operatorname{rad}(A)}{[A/\operatorname{rad}(A), A/\operatorname{rad}(A)]},
    \]
    since $\pi([A, A]) \subseteq [\pi(A), \pi(A)] = [A/\operatorname{rad}(A), A/\operatorname{rad}(A)]$.

    Since $K$ is algebraically closed, the Wedderburn-Artin Theorem gives:
    \[
    A/\operatorname{rad}(A) \cong M_{n_1}(K) \times \cdots \times M_{n_r}(K).
    \]
    By the Albert-Muckenhoupt-Shoda Theorem, $[M_{n_i}(K), M_{n_i}(K)] = \mathfrak{sl}_{n_i}(K)$, and thus $M_{n_i}(K) / [M_{n_i}(K), M_{n_i}(K)]$ is one-dimensional. Therefore,
    \[
    \dfrac{A/\operatorname{rad}(A)}{[A/\operatorname{rad}(A), A/\operatorname{rad}(A)]} \cong K^r.
    \]
    In particular, $\dim_K \left( \dfrac{A/\operatorname{rad}(A)}{[A/\operatorname{rad}(A), A/\operatorname{rad}(A)]} \right) = r$.

    Since $\overline{\pi}$ is a surjective linear map, we conclude that
    \[
    \dim_K \left( \dfrac{A}{[A, A]} \right) \geq r.
    \]
\end{proof}

\begin{remark}
The inequality in Proposition \ref{dimension bound} is not always an equality. For instance, in the algebra from Example \ref{example0}, we have $r = 2$ but $\dim_K A/[A, A] = 3$.
\end{remark}

\begin{remark}\label{remark sharp}
The main result of this paper, Theorem \ref{main theorem}, shows that for generalized block-triangular algebras, the set of commutators is exactly the set of multitrace zero elements. As a consequence, $\dim A / [A,A] = r$, i.e., for this class of algebras, the lower bound in Proposition \ref{dimension bound} is sharp. 
This provides a large class of algebras for which the natural map $\overline{\pi}$ is an isomorphism (of vector spaces).
\end{remark}

\begin{remark}
    In the paper \cite{AmitsurRowen}  the authors define the notion of reduced trace of an element in a finite dimensional central simple $K$-algebra $A$ (where $K$ is not necessarily algebraically closed), namely the reduced trace of $a\in A$ is the trace of the matrix corresponding to the element $a\otimes 1$ in $M_n(K)\otimes_K \overline K$. There, they prove that if $A$ has reduced trace zero then it is a sum of at most 2 commutators, and it is a commutator if $A$ is not a division ring. We can naturally define the notion of reduced multitrace for an arbitrary element on a finite dimensional unital $K$-algebra, for an arbitrary $K$ and investigate commutators in such algebras. Such problems will be investigated in future projects.
\end{remark}

\section*{Acknowledgments}

This work was supported by CNPq - Brazil (grant \#405779/2023-2) and  Fapesp - Brazil (grants \#2024/19129-8, \#2024/14914-9 and \#2025/02457-5).

\end{document}